\documentclass[a4paper, onesided]{amsart}
\usepackage{amsmath}
\usepackage{amsfonts}
\usepackage{amssymb}
\usepackage{graphicx}
\usepackage{verbatim}
\usepackage{amsaddr}

\newtheorem{theorem}{Theorem}[section]
\theoremstyle{plain}
\newtheorem{corollary}[theorem]{Corollary}

\newtheorem{lemma}[theorem]{Lemma}
\newtheorem{proposition}[theorem]{Proposition}

\newtheorem{note}[theorem]{Note}

\newcommand{\cl}{\overline}
\newcommand{\B}{\mathcal{B}}
\newcommand{\U}{\mathcal{U}}
\newcommand{\V}{\mathcal{V}}
\newcommand{\F}{\mathcal{F}}
\newcommand{\K}{\mathcal{K}}
\newcommand{\Po}{\mathcal{P}}
\newcommand{\m}{\mathfrak{m}}
\newcommand{\n}{\mathfrak{n}}

\long\def\symbolfootnote[#1]#2{\begingroup\def\thefootnote{\fnsymbol{footnote}}\footnote[#1]{#2}\endgroup}

\title{Topologies of (strong) uniform convergence on bornologies}
\author{L\!'ubica Hol\'a and Branislav Novotn\'y}
\address{Mathematical Institute, Slovak Academy of Sciences, \v Stef\'anikova 49, SK-814 73 Bratislava, Slovakia}
\email{lubica.hola@mat.savba.sk, branislav.novotny@mat.savba.sk}
\begin{document}
\maketitle
\begin{abstract}
	We continue the study of topologies of strong uniform convergence on bornologies initiated  in 
	[G.~Beer and S.~Levi, \newblock Strong uniform continuity, \newblock {\em J. Math Anal. Appl.}, 350:568--589, 2009]
	and
	[G.~Beer and S.~Levi, \newblock Uniform continuity, uniform convergence and shields, \newblock {\em Set-Valued and Variational Analysis}, 18:251--275, 2010].
	We study cardinal invariants of topologies of (strong) uniform convergence on bornologies on the space of continuous real-valued functions and we also generalize some known results from the literature.
\end{abstract}

\section{Introduction}\label{introduction}

\bigskip

	Topologies of strong uniform convergence on bornologies were introduced by Beer and Levi in their paper \cite{beerlevi1} and then studied also in 
	\cite{beerlevi2}, \cite{casdiho} and \cite{hola}.
Let $\B$ be a bornology in a metric space $(X,\rho)$, that is, a cover of $X$ that also forms an ideal. In \cite{beerlevi1} authors introduced the variational notions of strong uniform continuity of a function on $\B$ as an alternative to uniform continuity of the restriction of the function to each member of $\B$, and the topology of strong uniform convergence on $\B$ as an alternative to the classical topology of uniform convergence on $\B$. In \cite{beerlevi2} the authors continued this study, showing that shields play a pivotal role. Shields are successfully used also in our paper.

	In our paper we study topologies of (strong) uniform convergence on bornologies from the point of view of cardinal invariants (like character, cellularity, density, network weight, weight) extending some results from \cite{beerlevi1} and \cite{casdiho}.
	We use shields in our estimations of density and weight.
	
	Notice that cardinal invariants of $C_\alpha(X)$, the set-open topology on the space of continuous real-valued functions defined on a topological space $X$, were studied in [R.A. McCoy and I.~Ntantu, \newblock {\em Topological Properties of Spaces of Continuous Functions}, \newblock Springer-Verlag, 1988] only for a hereditarily closed compact network $\alpha$. Such topology is a special case of a topology of uniform convergence on a bornology with a compact base.
	Thus our results extend also some results from this book.

A system $\B$ of subsets of $X$ is called \emph{bornology} on $X$ if it fulfills the following properties:
\[(b1)\ \text{if }B_1\in\B \text{ and }B_2\subset B_1 \text{, then } B_2\in\B,\]
\[(b2)\ \text{if }B_1,B_2\in\B \text{, then } B_1\cup B_2\in\B,\]
\[(b3)\ \bigcup\B=X.\]

	Every system $\mathcal A$ generates the smallest bornology on $\bigcup\mathcal A$ containing $\mathcal A$.
A system $\B_0\subset\B$ is a base of a bornology $\B$ if $\B_0$ is cofinal in $\B$, with respect to $\subset$; i.e. for every $B\in\B$ there is $B_0\in\B_0$ with $B\subset B_0$.
	We will define the \emph{cofinality} of $\B$ by \[cf(\B)=\aleph_0+\min\{|\B_0|; \B_0\text{ is a base of }\B\}.\]

For a bornology $\B$ on a topological space $(X,\tau)$ we can introduce another condition:
\[(b4)\ B\in\B \text{, then } \cl{B}\in\B.\]
We will say that a bornology $\B$ has a \emph{closed (compact) base} iff it has a base consisting of closed (compact) sets. Note that $(b4)$ is equivalent to the condition that $\B$ has a closed base. The smallest bornology containing the bornology $\B$ fulfilling $(b4)$ is the one generated by the system $\{\cl{B}; B\in\B\}$, we will denote it $\cl{\B}$.

Now let $(X,\rho)$ and $(Y,\eta)$ be metric spaces and $\B$ be a bornology on $X$. Denote by $B^\delta$ an open $\delta$ enlargement of the set $B\subset X$.
We will be interested in the topologies $\tau_\B$ and $\tau^s_\B$ on $C(X,Y)$, the space of continuous functions from $X$ to $Y$, where $\tau_\B$ is generated by the uniformity with basic entourages of the form (see \cite{mccoy})
\[[B,\epsilon]=\{(f,g);\text{ for every }x\in B\ \eta(f(x),g(x))<\epsilon\}\quad B\in\B,\ \epsilon>0\]
and $\tau^s_\B$ is generated by the uniformity with basic entourages of the form
\[[B,\epsilon]^s=\{(f,g);\exists\delta>0 \text{ for every }x\in B^\delta\ \eta(f(x),g(x))<\epsilon\}\quad B\in\B,\ \epsilon>0.\]
Note that in both cases bornologies $\B$ and $\cl{\B}$ generate the same uniformities on $C(X,Y)$, so we will be interested only in bornologies with a closed base.

Examples of bornologies on a Hausdorff topological space $(X,\tau)$ with a closed base are $\F$ - the system of all finite subsets of $X$, $\Po$ - the system of all subsets of $X$, $\mathcal K$ - the system of all subsets of $X$ with a compact closure. For every bornology $\B$ on $X$ holds $\F\subset\B\subset\Po$.
Note that if $\B\subset\K$ (i.e. $\B$ has a compact base) then $\tau_\B=\tau_\B^s$ (see \cite[Corollary 6.6]{beerlevi1}).

\section{Cardinal Functions Depending on a Bornology}

For every cardinal number $\m$ denote by $\m^+$ its cardinal successor. Let $\omega=\omega_0\setminus\{0\}$ i.e. the set of positive integers.

Let $(X,\tau)$ be a Hausdorff space and $\B$ be a bornology on $X$ with a closed base.
Put $o(X)=|\tau|$.
The following notions are from \cite{mccoy}. We will stick to the names but use the notation which is more appropriate in this context.

A system $\gamma\subset\Po$ is a \emph{$\B$-network} on $X$ if for every $B\in\B$ and $U\in\tau$ such that $\cl{B}\subset U$ there is $C\in\gamma$ with $\cl{B}\subset C\subset U$. Define the \emph{$\B$-network weight} of $X$ by
$$nw(X,\B)=\aleph_0+\min\{|\gamma|;\gamma\text{ is a }\B\text{-network on }X\}$$
and the\emph{ weak $\B$-covering number} of $X$ by
$$d(X,\B)=\aleph_0+\min\{|\gamma|;\gamma\subset\B,\ \cl{\bigcup\gamma}=X\}.$$
A system $\U\subset\tau$ is an \emph{open $\B$-cover} of $X$ if for every $B\in\B$ there is $U\in\U$ such that $B\subset U$.
The \emph{$\B$-Lindel\"of degree} of $X$ is defined as
\begin{multline*}
L(X,\B)=\aleph_0+\min\{\m;\text{every open }\B\text{-cover of }X\text{ has an open }\B\text{-subcover of }X\\ \text{ with cardinality}\leq\m\}.
\end{multline*}
If we consider the bornology $\F$ the above mentioned cardinal invariants reduce to well known ones, particularly $nw(X,\F)=nw(X)$, the network weight of $X$, $d(X,\F)=d(X)$, the density of $X$ and $L(X,\F)=L(X)$, the Lindel\"of degree of $X$.
Remind that a topological space is called \emph{$\m$-compact} if every open cover of $X$ has an open subcover with cardinality $<\m$ (see \cite{juhasz}). Define \emph{compactness degree} of $X$ by
\[\delta(X)=\aleph_0+\min\{\m;X\text{ is }\m\text{-compact}\}\]
as suggested in \cite[1.8.]{juhasz}, and \emph{compactness degree} of $\B$ by
\[\delta(X,\B)=\sup\{\delta(\cl{B});B\in\B\}.\]
Note that $\delta(X,\Po)=\delta(X)$ and $L(X)\leq\delta(X)\leq L(X)^+$.
Similarly we can define
\[\delta_D(X)=\aleph_0+\min\{\m;\forall\text{closed discrete }D\subset X\text{ is }|D|<\m\},\]
\[\delta_D(X,\B)=\sup\{\delta_D(\cl{B});B\in\B\}=\aleph_0+\min\{\m;\forall\text{closed discrete }D\in\B\text{ is }|D|<\m\}.\]
Recall that $e(X)=\aleph_0+\sup\{|D|;D \text{ is a closed discrete subset of  } X\}$ is called an \emph{extent} of $X$ and observe that $e(X)\leq\delta_D(X)\leq e(X)^+$ and $\delta_D(X,\B)\leq\delta(X,\B)$.

Now let $(X,\rho)$ be a metric space and $\B$ be a bornology on $X$ with a closed base.
A system of open sets $\U$ is a \emph{strong open $\B$-cover} of $X$ if for every $B\in\B$ there is $U\in\U$ and $\delta>0$ such that $B^\delta\subset U$. The \emph{strong $\B$-Lindel\"of degree} of $X$ is defined as
\begin{multline*}
L^s(X,\B)=\aleph_0+\min\{\m;\text{every strong open }\B\text{-cover of }X\\ \text{ has a strong open }\B\text{-subcover of }X \text{ with cardinality}\leq\m\}.
\end{multline*}

A metric space $X$ is called \emph{compact in generalized sense (or GK)}, see \cite{barcos}, if every open cover of $X$ has a subcover with cardinality less than $d(X)$. The following proposition is a consequence of \cite[Theorem 7]{barcos}
\begin{proposition}\label{delta}
	Let $(X,\rho)$ be a metric space and $\B$ be a bornology with a closed base, then $\delta(X)=\delta_D(X)$, $\delta(X,\B)=\delta_D(X,\B)$, if $X$ is GK then $\delta(X)=d(X)$ and if $X$ is not GK then $\delta(X)=d(X)^+$.	
\end{proposition}

\begin{proposition}\label{bornology_invariants}
	Let $(X,\tau)$ be a Hausdorff space and $\B$ be a bornology with a closed base, then $L(X,\B)\leq\min\{cf(\B),nw(X,\B)\}\leq o(X)$ and in the case of a metric space $(X,\rho)$ we also have $d(X,\B)\leq L^s(X,\B)\leq L(X,\B)$.
\end{proposition}
\begin{proof}
	Let $\U$ be an arbitrary open $\B$-cover of $X$, i.e. for every $B\in\B$ there is $U_B\in\U$ such that $B\subset U_B$. If $\B_0$ is a base of $\B$ then the system $\{U_B;B\in\B_0\}$ is an open $\B$-subcover of $X$, so $L(X,\B)\leq cf(\B)$. If $\gamma$ is a $\B$-network on $X$ with $|\gamma|\leq nw(X,\B)$ then for every closed $B\in\B$ we have $C_B\in\gamma$ such that $B\subset C_B\subset U_B$. Let $\gamma_0=\{C_B;B\in\B,\ B\text{ closed}\}$ thus for every $C\in\gamma_0$ there is $V_C\in\U$ with $C\subset V_C$. The system $\{V_C;C\in\gamma_0\}$ is an open $\B$-subcover of $X$ and since $|\gamma_0|\leq|\gamma|$ we have $L(X,\B)\leq nw(X,\B)$. Inequality $cf(\B)\leq o(X)$ follows from the fact that $\B$ has a closed base and $nw(X,\B)\leq o(X)$ is trivial.
	
	Now suppose that $(X,\rho)$ is a metric space. For the first inequality observe that for every $n\in\omega$ the system $\{B^{1/n}; B\in\B\}$ is a strong open $\B$-cover of $X$ hence there is $\B_n\subset\B$ such that $|\B_n|\leq L^s(X,\B)$ and $\{B^{1/n}; B\in\B_n\}$ is a strong open $\B$-cover of $X$ so $X=\bigcup\{B^{1/n}; B\in\B_n\}$. Now put $\B_0=\bigcup_{n\in\omega}\B_n$ and we have $X=\cl{\bigcup\B_0}$. For the second inequality take an arbitrary $\U$ a strong open $\B$-cover of $X$, i.e. for every $B\in\B$ there is $\delta_B>0$ and $U_B\in\U$ such that $B^{2\delta_B}\subset U_B$. Since $\{B^{\delta_B};B\in\B\}$ is an open $\B$-cover of $X$ there is $\B_0\subset\B$ such that $|\B_0|\leq L(X,\B)$ and $\{B^{\delta_B};B\in\B_0\}$ is an open $\B$-cover of $X$, hence $\{U_B;B\in\B_0\}$ is a strong open $\B$-cover of $X$.
\end{proof}
\begin{proposition}\label{nwx}
	Let $(X,\tau)$ be a Hausdorff space and $\B$ be a bornology with closed base, then 
	$2^{<\delta_D(X,\B)}\leq nw(X,\B)\leq w(X)^{<\delta(X,\B)}$.
\end{proposition}
\begin{proof}
	
	For the first inequality take an arbitrary $\m<\delta_D(X,\B)$ and an arbitrary $\gamma$ a $\B$-network in $X$. There is a closed discrete $D\in\B$ with $|D|=\m$. For every $A\subset D$ there is $C_A\in\gamma$ such that $A\subset C_A\subset X\setminus(D\setminus A)$. Note that for every $A,B\subset D$ hold $C_A\not=C_B$ whenever $A\not=B$, hence $|\gamma|\geq 2^\m$.
	
	For the second inequality take a base $\U$ on $X$ with $|\U|=w(X)$. Let $B\in\B$ and $U\in\tau$ such that $\cl{B}\subset U$. For every $x\in\cl{B}$ there is $U_x\in\U$ with $x\in U_x\subset U$. The system $\{U_x;x\in\cl{B}\}$ is an open cover of $\cl{B}$ so it has an open subcover with cardinality less than $\delta(X,\B)$. So the system $\{\cup\U_0;\U_0\subset\U,|\U_0|<\delta(X,\B)\}$ is the $\B$-network on $X$.
\end{proof}

\section{Cardinal Invariants on $C(X)$}

	Let $(X,\rho)$ be a metric space and let $\B$ be a bornology with a closed base on $X$. Let $C(X)$ be the space of continuous real-valued functions on $X$. Let $\tau_\B$ and $\tau^s_\B$ be defined as in the section \ref{introduction}.

	For every $B\in\B$ and $\epsilon>0$ holds $[B,\epsilon]^s\subset[B,\epsilon]$, so $\tau_\B\subset\tau^s_\B$. Note that $\tau_\Po=\tau^s_\Po$ is the topology of uniform convergence, but since the notation suggests the topology of pointwise convergence we will use $\tau_U$ instead of $\tau_\Po$. Cardinal invariants on $C(X)$ will be denoted by $\varphi(\tau)$, instead of $\varphi(C(X),\tau)$, or $\varphi(f,\tau)$ for point specific functions, where $f\in C(X)$ and $\tau$ is a topology on $C(X)$.

	We will use the standard notation for usual cardinal invariants, so
\newline
$\chi,\pi\chi,\psi,\Delta,t,w,nw,u,c,d,hd,hL$, denote \emph{character,$\pi-$character, pseudo character, diagonal degree, tightness, weight, network weight, uniform weight, cellularity, density, hereditary density and hereditary Lindel\"of degree}, respectively, see \cite{engelking}, \cite{juhasz}.

	Observe that $(C(X),\tau_\B)$, $(C(X),\tau^s_\B)$ are topological groups. For every topological group $G$ and for every $g\in G$ holds $\chi(g,G)=\chi(G)=\pi\chi(G)=\pi\chi(g,G)$ and $\psi(g,G)=\psi(G)=\Delta(G)$.
\begin{lemma}\label{char}
	If for some $B_1,B_2\in\B;\ \epsilon_1,\epsilon_2>0;\ f_1,f_2\in C(X)$ holds $[B_1,\epsilon_1]^s(f_1)\subset [B_2,\epsilon_2](f_2)$, then $B_2\subset\cl{B_1}$.
\end{lemma}
\begin{proof}
	Suppose there is $x_0\in B_2\setminus\cl{B_1}$, then there is $\delta>0$ such that $x_0\not\in\cl{B_1^\delta}$. There is a function $g\in C(X)$ such that for $x\in\cl{B_1^\delta}$ is $g(x)=f_1(x)$ and $g(x_0)=f_2(x_0)+2\epsilon_2$ and so $g\in[B_1,\epsilon_1]^s(f_1)\setminus [B_2,\epsilon_2](f_2)$.
\end{proof}
\bigskip

The following theorem generalizes \cite[Theorem 7.1]{beerlevi1} for $(C(X),\tau_\B^s)$.
\begin{theorem}\label{character}
	$\pi\chi(\tau_\B)=\pi\chi(\tau^s_\B)=\chi(\tau_\B)=\chi(\tau^s_\B)=u(\tau_\B)=u(\tau^s_\B)=cf(\B)$.
\end{theorem}
\begin{proof}
	
 The inequality $\pi\chi(\tau)\leq\chi(\tau)\leq u(\tau)$ holds generally for any completely regular space, inequalities $u(\tau_\B)\leq cf(\B)$ and $u(\tau^s_\B)\leq cf(\B)$ follow directly from the definition of the respective topologies.

	To see that $cf(\B)\leq\pi\chi(0,\tau_\B)$ take a $\pi-$base $\U$ at $0$ with cardinality $\pi\chi(0,\tau_\B)$. For every $U\in\U$ take $B_U\in\B$, $\epsilon_U>0$ and $f_U\in C(X)$ such that $[B_U,\epsilon_U](f_U)\subset U$. Since $\U$ is a $\pi-$base at $0$, for every $B\in B$ there is $U\in\U$ such that $U\subset [B,1](0)$. Then $[B_U,\epsilon_U](f_U)\subset [B,1](0)$, so by Lemma \ref{char} holds  $B\subset\cl{B_U}$ and hence the system $\{\cl{B_U};U\in\U\}$ is a base of $\B$. Similarly for $cf(\B)\leq\pi\chi(0,\tau^s_\B)$.
\end{proof}

\begin{theorem}\label{pseudo}
	$\psi(\tau_\B)=\psi(\tau^s_\B)=\Delta(\tau_\B)=\Delta(\tau^s_\B)=d(X,\B)$.
\end{theorem}
\begin{proof}
	We have that $\psi(\tau)\leq\Delta(\tau)$ for any topological space. Since $\tau_\B\subset\tau^s_\B$ then $\psi(\tau^s_\B)\leq\psi(\tau_\B)$.

	Now take $\B_0\subset\B$ such that $\cl{\bigcup\B_0}=X$ and $|B_0|=d(X,\B)$. We want to prove that $A:=\bigcap\{[B,1/n];B\in\B_0,\,n\in\omega\}\subset\Delta_{C(X)}$, where $\Delta_{C(X)}=\{(f,f);f\in C(X)\}$ is the diagonal of the space $C(X)$. Suppose $(g,h)\in A$ and $x\in B\in\B_0$. Then for every $n\in\omega$ holds $|f(x)-g(x)|<1/n$ so $f(x)=g(x)$ and since $\cl{\bigcup\B_0}=X$ and $g,h\in C(X)$, then $g=h$. Now observe that $\Delta_{C(X)}\subset int([B,1/n])\subset[B,1/n]$ (interior with respect to $\tau_\B\times\tau_\B$) and also $\Delta_{C(X)}\subset int([B,1/n]^s)\subset[B,1/n]^s\subset[B,1/n]$ (interior with respect to $\tau^s_\B\times\tau^s_\B$) and hence $\Delta(\tau_\B)\leq d(X,\B)$ and $\Delta(\tau^s_\B)\leq d(X,\B)$.

	To prove that $d(X,\B)\leq\psi(0,\tau^s_\B)$ take a system $\U$ of open sets such that $\bigcap\U=\{0\}$. For every $U\in\U$ take $B_U\in\B$ and $\epsilon_U>0$ such that $[B_U,\epsilon_U]^s(0)\subset U$. Put $\B_0=\{B_U;U\in\U\}$, we will show that $\cl{\bigcup\B_0}=X$ and we are done. Suppose there is $x_0\in X\setminus\cl{\bigcup\B_0}$, then there is $\delta>0$ such that $x_0\not\in\cl{(\bigcup\B_0)^\delta}$ and there is a function $g\in C(X)$ such that for $x\in\cl{(\bigcup\B_0)^\delta}$ is $g(x)=0$ and $g(x_0)\not=0$ and therefore $g\in\bigcap\U$ contrary to supposition.
\end{proof}
\begin{proposition}\label{zero}
	Let $\F\subset C(X)$. Then $0\in\cl{\F}$ if and only if for every $\epsilon>0$ the system $\{f^{-1}(-\epsilon,\epsilon);f\in\F\}$ is an open $\B$-cover (a strong open $\B$-cover) of $X$, where $\cl{\F}$ is the closure with respect to $\tau_\B$ ($\tau^s_\B$).
\end{proposition}
\begin{proof}
	Proofs for $\tau_\B$ and $\tau^s_\B$ are very similar, so we will show it only for $\tau_\B$. The statement $0\in\cl{\F}$ is equivalent to the statement that for every $\epsilon>0$ and for every $B\in\B$ there is $f\in\F\cap[B,\epsilon](0)$. Hence $B\subset f^{-1}(-\epsilon,\epsilon)$ and the rest follows.
\end{proof}

\begin{theorem}\label{tightness}
	$t(\tau_\B)=L(X,\B)$ and $t(\tau^s_\B)=L^s(X,\B)$.
\end{theorem}
\begin{proof}
	Since the proofs for $\tau_\B$ and $\tau^s_\B$ are similar we will prove it only for $\tau^s_\B$ and it suffices to prove that $t(0,\tau^s_\B)=L^s(X,\B)$.

	For $L^s(X,\B)\leq t(0,\tau^s_\B)$ take a strong open $\B$-cover $\U$ of $X$. For every $B\in\B$ there are $\delta_B>0$ and $U_B\in\U$ such that $B^{2\delta_B}\subset U_B$ and so we can take $f_B\in C(X)$ such that $f_B(\cl{B^{\delta_B}})=\{0\}$ and $f_B(X\setminus U_B)=\{1\}$. By Proposition \ref{zero} is $0\in\cl{\{f_B;B\in\B\}}$ and so there is $\B_0\subset\B$ such that $|\B_0|\leq t(0,\tau^s_\B)$ and $0\in\cl{\{f_B;B\in\B_0\}}$. Since $f_B^{-1}(-1/2,1/2)\subset U_B$,by Proposition \ref{zero} the system $\{U_B;B\in\B_0\}$ is a strong open $\B$-cover of $X$.

	To prove $t(0,\tau^s_\B)\leq L^s(X,\B)$ take any $\F\subset C(X)$ with $0\in\cl{\F}$. By Proposition \ref{zero}, for every $n\in\omega$ the system $\{f^{-1}(-1/n,1/n);f\in\F\}$ is a strong open $\B$-cover of $X$ and thus there is $\F_n\subset\F$ such that $|\F_n|\leq L^s(X,\B)$ and $\{f^{-1}(-1/n,1/n);f\in\F_n\}$ is a strong open $\B$-cover of $X$. Put $\F_0:=\bigcup_{n\in\omega}\F_n\subset\F$ and we have that $\F_0\leq L^s(X,\B)$ and $0\in\cl{\F_0}$.
\end{proof}
\begin{theorem}\label{nwcx}
	$nw(X,\B)\leq nw(\tau_\B)\leq w(X)^{<\delta(X,\B)}$.
\end{theorem}
\begin{proof}
	For the first inequality let $\F$ be a network in $C(X)$ with $|\F|=nw(\tau_\B)$. For every $F\in\F$ put $F^*=\{x\in X;f(x)>0\text{ for every }f\in F\}$ and put $\F^*=\{F^*;F\in\F\}$. To show that $\F^*$ is a $\B-$network for $X$, let $B\in\B$ and $U$ be an open set such that $\cl{B}\subset U$. Take $f\in C(X)$ such that $f(\cl{B})=\{1\}$ and $f(X\setminus U)=\{0\}$ and $F\in\F$ such that $f\in F\subset [\cl{B},1](f)$. Now $\cl{B}\subset F^*\subset U$.
	
	For the second inequality let $\alpha$ be a countable base on $R$ and $\beta$ be a base on $X$ with $|\beta|=w(X)$. For every $I\in\alpha$ and $U\in\beta$ denote $\U_{I,U}=\{g\in C(X);g(U)\subset I\}$. Fix an arbitrary $f\in C(X)$ and $B\in\B$ and $\epsilon > 0$. For every $x \in X$ there is $I_x\in\alpha$ such that $f(x)\in I_x$, $I_x\subset(f(x)-\epsilon/2,f(x)+\epsilon/2)$ and there is $U_x\in\beta$ such that $x\in U_x\subset f^{-1}(I_x)$. Let $\V_x=\U_{I_x,U_x}$. Then $f\in\V_x$ and for every $g\in\V_x$ and for every $x'\in U_x$ holds $|f(x')-g(x')|<\epsilon$. Since $\{U_x;x\in X\}$ is an open cover of $X$ then $\cl{B}\subset\bigcup\{U_{x_k};k<\n\}$ for some $\n<\delta(X,\B)$ and therefore $f\in\bigcap\{\V_{x_k};k<\n\}\subset[B,\epsilon](f)$. Thus the system $\{\bigcap\{\U_{I_k,U_k};k<\n\};I_k\in\alpha,U_k\in\beta,\n<\delta(X,\B)\}$ is a network for $\tau_\B$ with cardinality $w(X)^{<\delta(X,\B)}$.
\end{proof}
From Proposition \ref{nwx}, Theorem \ref{nwcx} and the fact that $nw(X,\B)\geq nw(X,\F)=nw(X)=w(X)$ follows:
\begin{corollary}\label{nwcxcor}
	If $w(X)<2^{<\delta(X,\B)}$ then $nw(X,\B)=nw(\tau_\B)=2^{<\delta(X,\B)}$; and if $w(X)=2^\n\geq 2^{<\delta(X,\B)}$ or $\delta(X,\B)=\aleph_0$ then $nw(X,\B)=nw(\tau_\B)=w(X)$.
\end{corollary}
\begin{theorem}\label{network weight}
	$nw(\tau_\B)=nw(X,\B)d(\tau_B)$ and $nw(\tau_\B^s)=nw(X,\B)d(\tau_B^s)$.
\end{theorem}
\begin{proof}
	From Theorem \ref{nwcx} and $\tau_\B\subset\tau_B^s$ we have that $nw(\tau_\B^s)\geq nw(\tau_\B)\geq nw(X,\B)$. It remains to show that $nw(\tau_\B)\leq nw(X,\B)d(\tau_B)$ and $nw(\tau_\B^s)\leq nw(X,\B)d(\tau_B^s)$. Since the proofs are very similar we will do it only for $\tau_\B$. Let $\gamma$ be a $\B-$network for $X$ with $|\gamma|=nw(X,\B)$ and $D$ be a dense subset of $C(X)$ with $|D|=d(\tau_\B)$. We will show that the system $\{[C,1/n](g);C\in\gamma,n\in\omega,g\in D\}$ is a network for $(C(X),\tau_\B)$. For every $f\in C(X)$, $n\in\omega$ and $B\in\B$ there is $g\in D\cap[B,1/2n](f)$. There is an open $U\supset\cl{B}$ such that $g\in[U,1/n](f)$. Choose $C\in\gamma$ with $\cl{B}\subset C\subset U$ and thus we have $f\in[C,1/n](g)\subset[B,2/n](f)$ which concludes the proof.
\end{proof}
\begin{note}
	Concerning the topology $\tau_\B$ in Theorems \ref{character}, \ref{pseudo} it is enough to suppose that $X$ is a Tychonoff space and in Theorems \ref{tightness}, \ref{nwcx}, \ref{network weight}, that $X$ is Tychonoff such that for every closed $A\subset X$ and every closed $B\in\B$ there is a continuous function $f:X\to R$ such that $f(A)\subset\{0\}$ and $f(B)\subset\{1\}$.
	 Therefore these Theorems generalize \cite[Theorems 4.4.1, 4.3.1, 4.7.1]{mccoy}
\end{note}
In what follows the suppositions from the beginning of this section that $(X,\rho)$ is a metric space and $\B$ has a closed base are mandatory.

From \cite{costantini} we have the following result.
\begin{proposition}\label{basic_density}
	$d(\tau_U)=2^{<\delta(X)}$.
\end{proposition}

\begin{theorem}\label{weak_density}
	$2^{<\delta(X,\B)}\leq c(\tau_\B)\leq d(\tau_\B)\leq cf(\B)2^{<\delta(X,\B)}$.
\end{theorem}
\begin{proof}
	For the first inequality take an arbitrary $\m<\delta(X,\B)$, then there is a closed $B\in\B$ with $\m<\delta(B)$, i.e. there is a closed discrete $D\subset B$ with $|D|=\m$. Every function $g\in 2^D=\{0,1\}^D$ can be extended to $f_g\in C(X)$. The system $\{int([B,1/2](f_g));g\in 2^D\}$ is cellular in $(C(X),\tau_\B)$, so $2^\m\leq c(\tau_\B)$.

	The second inequality is general. To prove the last inequality we can by Proposition \ref{basic_density} take for every closed $B\in\B$ the system $\F_B\subset C(B)$ which is dense in $(C(B),\tau_U)$ and $|\F_B|\leq2^{<\delta(B)}\leq2^{<\delta(X,\B)}$. Let $\F^*_B\subset C(X)$ be the system of extensions of functions from $\F_B$ and put $\F^*=\bigcup_{B\in\B_0}\F^*_B$, where $\B_0$ is a closed base of $\B$ with $|\B_0|\leq cf(\B)$. One can easily see that $|\F^*|\leq cf(\B)2^{<\delta(X,\B)}$ and $\F^*$ is dense in $(C(X),\tau_\B)$.
\end{proof}

Since for topological groups  $w(\tau)=d(\tau)\chi(\tau)$, we have the following result.
\begin{corollary}\label{weak_base}
	$w(\tau_\B)=cf(\B)2^{<\delta(X,\B)}$.
\end{corollary}

We will need the following notion that was introduced in \cite{beerlevi2}. We say that a set $S\subset X$ is a \emph{shield} for a set $A\subset X$, iff for every closed $C\subset X$ such that $C\cap S=\emptyset$ there is an $\epsilon>0$ such that $C\cap A^\epsilon=\emptyset$. We say that a family $\mathcal A$ is \emph{shielded from closed sets} provided each $A\in\mathcal A$ has a shield in $\mathcal A$. As one of main results in \cite{beerlevi2} it was proved that $\tau_\B^s$ and $\tau_\B$ coincide on $C(X)$ iff $\B$ is shielded from closed sets.

Let us now define some variants of compactness degree.
\[\delta_0^s(B)=\aleph_0+\min\{\m;\forall\text{closed and discrete }D\subset X\ \exists\epsilon>0:|D\cap B^\epsilon|<\m\},\text{ for } B\subset X,\]
\[\delta_0^s(X,\B)=\sup\{\delta_0^s(\cl{B});B\in\B\},\]
\[\delta_1^s(B)=\aleph_0+\min\{\delta(S); S\text{ is a closed shield for }B\}\]
\[\delta_1^s(X,\B)=\sup\{\delta_1^s(\cl{B});B\in\B\}.\]
\begin{proposition}
For every $B$ holds $\delta(B)\leq\delta_0^s(B)\leq\delta_1^s(B)$ and hence $\delta(X,\B)\leq\delta_0^s(X,\B)\leq\delta_1^s(X,\B)$. Moreover of $\B$ is shielded from closed sets then $\delta_1^s(X,\B)\leq\delta(X,\B)$.
\end{proposition}
\begin{proof}
First inequality follows from Proposition \ref{delta}. For the second suppose that $\kappa<\delta_0^s(B)$ then there is a closed discrete $D\subset X$ such that for every $\epsilon>0$ holds $|D\cap B^\epsilon|\geq\kappa$. Let $S$ be a closed shield for $B$. Since $D\setminus S$ is closed and disjoint from $S$ there is $\epsilon>0$ such that $(D\setminus S)\cap B^\epsilon=\emptyset$
i.e. $D\cap B^\epsilon\subset D\cap S$ therefore $\delta(S)>|D\cap S|\geq|D\cap B^\epsilon|\geq\kappa$ and hence $\delta_1^s(B)>\kappa$. Finally $\delta_0^s(B)\leq\delta_1^s(B)$.

Now suppose that $\B$ is shielded from closed sets. For every closed $B\in\B$ there is $B'\in\B$, a closed shield for $B$, hence $\delta_1^s(B)\leq\delta(B')\leq\delta(X,\B)$ and therefore  $\delta_1^s(X,\B)\leq\delta(X,\B)$.
\end{proof}
\begin{theorem}
	$2^{<\delta_0^s(X,\B)}\leq c(\tau^s_\B)$.
\end{theorem}
\begin{proof}
	For every $\m<\delta_0^s(X,\B)$ there is closed $B\in\B$ with $\m<\delta_0^s(B)$, so there is closed discrete $D\subset X$ such that for every $\epsilon>0$, $|D\cap B^\epsilon|\geq\m$.
	For every $n\in\omega$ denote $D_n=D\cap B^{1/n}$, $D'_n=D_n\setminus D_{n+1}$ and $D_\omega=D\cap B$. We have  $D_n=\bigcup_{k\geq n}D'_k\cup D_\omega$.
	Let us consider following three cases:
	
	\textit{Case 1:} $|D_\omega|\geq\m$. Every function $g\in 2^{D_\omega}=\{0,1\}^{D_\omega}$ has its extension $f_g\in C(X)$. The system $\{int([B,1/2]^s(f_g));g\in 2^{D_\omega}\}$ is cellular in $(C(X),\tau^s_\B)$ so $2^\m\leq c(\tau^s_\B)$.
	
	\textit{Case 2:} $|D_\omega|<\m$ and $cf(\m)>\aleph_0$. For every $n\in\omega$ there is $k_n\geq n$ such that $|D'_{k_n}|\geq\m$. We can suppose that $k_{n+1}>k_n$ so we have $D'_{k_n}\supset\{x^n_\gamma;\gamma<\m\}$ such that $x^n_\gamma=x^m_\sigma$ iff $m=n$ and $\gamma=\sigma$. For every $g\in 2^\m$ there is $f_g\in C(X)$ such that $f_g(x^n_\gamma)=g(\gamma)$ for every $n\in\omega$ and $\gamma<\m$. The system $\{int([B,1/2]^s(f_g));g\in 2^\m\}$ is cellular in $(C(X),\tau^s_\B)$.
	
	\textit{Case 3:} $|D_\omega|<\m$ and $cf(\m)=\aleph_0$. There is a sequence $\{\m_n<\m;n\in\omega\}$ such that $\m_{n+1}>\m_n$ and $\sup\{\m_n;n\in\omega\}=\m$. For every $n\in\omega$ there is $k_n\geq n$ such that $|D'_{k_n}|\geq\m_n$. We can suppose that $k_{n+1}>k_n$ so we have $D'_{k_n}\supset\{x^n_\gamma;\gamma<\m_n\}$ such that $x^n_\gamma=x^m_\sigma$ iff $m=n$ and $\gamma=\sigma$. For every $g\in 2^\m$ there is $f_g\in C(X)$ such that $f_g(x^n_\gamma)=g(\gamma)$ for every $n\in\omega$ and $\gamma<\m_n$. The system $\{int([B,1/2]^s(f_g));g\in 2^\m\}$ is cellular in $(C(X),\tau^s_\B)$.
\end{proof}

\begin{theorem}\label{strong_density2}
	$d(\tau^s_\B)\leq cf(\B)2^{<\delta_1^s(X,\B)}$
\end{theorem}
\begin{proof}
Let $\B_0$ be a closed base of $\B$ such that $|\B_0|=cf(\B)$. For every $B\in\B_0$ there is a closed shield $S_B$ with $\delta(S_B)=\delta_1^s(B)$. Let $\F_B$ be dense in $(C(S_B),\tau_U)$ with $|\F_B|\leq2^{<\delta(S_B)}$. For every $f\in\F_B$ there is $f^*:X\to R$ a continuous extension of $f$. Let $\F^*_B=\{f^*;f\in\F_B\}$.  Let $\F=\bigcup_{B\in\B_0}\F^*_B$ then $|\F|\leq cf(\B)2^{<\delta_1^s(X,\B)}$. We will prove that $\F$ is dense in $(C(X),\tau_\B^s)$. Take an arbitrary basic open set of the form $[B,\epsilon]^s(g)$, where $B\in\B_0$, $\epsilon>0$ and $g\in C(X)$. There is $f\in\F^*_B$ such that for every $x\in S_B$ holds $|f(x)-g(x)|<\epsilon$ hence the set $A=\{x;|f(x)-g(x)|\geq\epsilon\}$ is closed and disjoint from $S_B$ and therefore there is a $\delta>0$ such that $B^\delta\cap A=\emptyset$ i.e. $f\in[B,\epsilon]^s(g)$.



\end{proof}
\begin{corollary}
	$cf(\B)2^{<\delta_0^s(X,\B)}\leq w(\tau^s_\B)\leq cf(\B)2^{<\delta_1^s(X,\B)}$.
\end{corollary}
Finally let us discuss cases when cardinal invariants are countable. We will need the following lemmas (in which we again leave the assumption that $X$ is a metric space). The first is a consequence of the main result in \cite{vidossich} and \cite[Corollary 4.2.2]{mccoy}. Recall that a space is \emph{submetrizable}, if it has a coarser metrizable topology.
\begin{lemma}\label{compdens}
	Let $X$ be a Tychonoff space. The following are equivalent.
	\begin{enumerate}
		\item $d(\tau_\F)=\aleph_0$,
		\item $d(\tau_\K)=\aleph_0$,
		\item $X$ has a coarser separable metrizable topology,
		\item $X$ is submetrizable and $d(X)\leq 2^{\aleph_0}$.
	\end{enumerate}
\end{lemma}
\begin{lemma}\label{compcel}\cite[4.9.2a]{mccoy}
	Let $X$ be a Tychonoff submetrizable space, then $c(\tau_\K)=\aleph_0$.
\end{lemma}
\begin{theorem}\cite{casdiho}
	Let $(X,\rho)$ be a metric space.
	\begin{enumerate}
		\item $c(\tau_\B)=\aleph_0$ iff  $c(\tau_\B^s)=\aleph_0$ iff $\B$ has a compact base;
		\item $w(\tau_\B)=\aleph_0$ iff  $w(\tau_\B^s)=\aleph_0$ iff $\B$ has a countable compact base;
		\item $nw(\tau_\B)=\aleph_0$ iff  $nw(\tau_\B^s)=\aleph_0$ iff $hd(\tau_\B)=\aleph_0$ iff  $hd(\tau_\B^s)=\aleph_0$ iff $hL(\tau_\B)=\aleph_0$ iff  $hL(\tau_\B^s)=\aleph_0$ iff $\B$ has a compact base and $X$ is separable;
		\item $d(\tau_\B)=\aleph_0$ iff  $d(\tau_\B^s)=\aleph_0$ iff $\B$ has a compact base and $X$ has a coarser separable metrizable topology.
	\end{enumerate}
\end{theorem}
\begin{proof}
	\quad
	\begin{enumerate}
		\item From Theorem \ref{weak_density} we have that $\delta(X,\B)\leq 2^{<\delta(X,\B)}\leq c(\tau_\B)$ so if $c(\tau_\B)=\aleph_0$ then $\delta(X,\B)=\aleph_0$; i.e. $\B$ has a compact base.
		On the other hand if $\B$ has a compact base, then $\tau_\B=\tau_\B^s\subset\tau_\K$. From Lemma \ref{compcel} we have that $c(\tau_\K)=\aleph_0$ and so  $c(\tau_\B^s)=\aleph_0$. The rest follows from $c(\tau_\B)\leq c(\tau_\B^s)$.
		\item From Corollary \ref{weak_base} we have that $cf(\B)2^{<\delta(X,\B)}=w(\tau_\B)$ so if $w(\tau_\B)=\aleph_0$ then $cf(\B)=\delta(X,\B)=\aleph_0$; i.e. $\B$ has a countable compact base.
		On the other hand if  $\B$ has a countable compact base, then $\tau_\B=\tau_\B^s$ so $w(\tau_\B^s)=w(\tau_\B)=cf(\B)2^{<\delta(X,\B)}=\aleph_0$. The rest follows from $w(\tau_\B)\leq w(\tau_\B^s)$.
		\item We know that $c\psi\leq hL\leq nw$ and $ct\leq hd\leq nw$ (,see \cite[3.12.7]{engelking}) so all cardinal invariants in (3) are by Proposition \ref{bornology_invariants}, Theorems \ref{pseudo}, \ref{tightness} and \ref{weak_density} greater or equal to $d(X,\B)\delta(X,\B)$. If $d(X,\B)\delta(X,\B)=\aleph_0$ then $\B$ has a compact base so $\B\subset\K$ and so $\aleph_0=d(X,\B)\geq d(X,\K)=d(X)$; i.e. $X$ is separable.
		All of the cardinal invariants in (3) are also less or equal to $nw(\tau_\B^s)$. If $\B$ has a compact base then by Corollary \ref{nwcxcor} we have that $nw(\tau_\B^s)=nw(\tau_\B)=w(X)=d(X)$. Thus if $X$ is separable then $nw(\tau_\B^s)=\aleph_0$. The rest follows from the mentioned inequalities.
		\item If $d(\tau_\B)=\aleph_0$ then by (1) we have that $\B$ has a compact base. Since $d(\tau_\F)\leq d(\tau_\B)$ we have from Lemma \ref{compdens} that $X$ has a coarser separable metrizable topology.
		On the other hand if $\B$ has a compact base then $\tau_\B^s=\tau_\B\subset\tau_\K$. If $X$ has a coarser separable metrizable topology then by \ref{compdens} we have that $d(\tau_\K)=\aleph_0$ and thus $d(\tau_\B^s)=\aleph_0$. The rest follows from $d(\tau_\B)\leq d(\tau_\B^s)$.
	\end{enumerate}
\end{proof}
\textbf{Acknowledgements.} Both authors were supported by VEGA 2/0047/10 and  APVV-269-11.
\bibliographystyle{plain}
\bibliography{bornology}
\end{document}